\documentclass[11pt]{article}

\usepackage{cite}
\usepackage{amssymb}
\usepackage{amsmath}
\usepackage{amsfonts}
\usepackage{amsthm}
\usepackage{fancyhdr}
\usepackage[explicit]{titlesec}
\usepackage{titletoc}
\usepackage{booktabs}
\usepackage[all]{xy}
\usepackage[colorlinks=black,linkcolor=red]{hyperref}
\usepackage{enumerate}
\usepackage{cancel}
\usepackage{color}
\usepackage{pgf,tikz}

\usepackage{url}

\usepackage{hyperref}

\newtheorem*{theorem*}{Theorem}
\newtheorem{theorem}{Theorem}[section]

\newtheorem{definition}{Definition}[section]

\newtheorem{lemma}{Lemma}[section]
\newtheorem{claim}{Claim}[section]
\newtheorem{problem}{Problem}[section]
\newtheorem{remark}{Remark}[section]

\numberwithin{equation}{section}

\everymath{\displaystyle}

\begin{document}

\title{Invariant random subgroups and action versus representation maximality}
\author{Peter J. Burton and Alexander S. Kechris}
\date{}
\maketitle

\renewcommand{\thefootnote}{}

\footnote{Research partially
supported by NSF Grant DMS-1464475}


\renewcommand{\thefootnote}{\arabic{footnote}}
\setcounter{footnote}{0}

\section{Introduction}
Let $G$ be a countably infinite group and $(X,\mu)$ a standard non-atomic probability space. We denote by $A(G,X, \mu)$ the space of measure preserving actions of $G$ on $(X,\mu)$ with the weak topology. If $\mathbf{a}, \mathbf{b} \in A(G,X, \mu)$, we say that $\mathbf{a}$ is $\textbf{weakly contained}$ in $\mathbf{b}$, in symbols $\mathbf{a}\preceq  \mathbf{b}$, if $\mathbf{a}$ is in the closure of the set of isomorphic copies of 
$\mathbf{b}$ (i.e., it is in the closure of the orbit of $\mathbf{b}$ under the action of the automorphism group of $(X,\mu)$ on $A(G,X, \mu)$; see \cite{K}). 
We say that $\mathbf{a} \in A(G,X, \mu)$ is $\textbf{action-maximal}$ if for all $\mathbf{b} \in A(G,X, \mu)$ we have $\mathbf{b}\preceq \mathbf{a}$. Such $\mathbf{a}$
exist by a result of Glasner-Thouvenot-Weiss, Hjorth, see \cite[Theorem 10.7]{K}).

Now let $H$ be a separable, infinite-dimensional Hilbert space and denote by $\textrm{Rep}(G,H)$ the space of unitary representations of $G$ on $H$ with the weak topology (see \cite[Appendix H]{K}). For $\pi, \rho \in \textrm{Rep}(G,H)$ we denote by $\pi\preceq\rho$ the usual relation of $\textbf{weak containment}$ of representations (see \cite{BHV}, \cite[Appendix H]{K}). We say that $\pi\in \textrm{Rep}(G,H)$ is $\textbf{representation-maximal}$ if for all $\rho \in \textrm{Rep}(G,H)$ we have $\rho\preceq \pi$. It is easy to check that such $\pi$ exist.

For any action $\mathbf{a}\in A(G,X, \mu)$, let $\kappa^{\mathbf{a}}$ be the associated representation on $L^2 (X, \mu)$, called the $\textbf{Koopman representation}$, and by $\kappa_0^{\mathbf{a}}$ its restriction to the orthogonal of the constant functions (see \cite[page 66]{K}). Then we have 
\[
\mathbf{a}\preceq  \mathbf{b} \implies \kappa_0^{\mathbf{a}}\preceq \kappa_0^{\mathbf{b}}
\]
but the converse fails, see \cite[pages 66 and 68]{K} and also \cite[page 155]{CK} for examples. However in all these examples the actions $\mathbf{a}, \mathbf{b}$ were not both ergodic and this led to the following question.

\begin{problem}\label{prob}
If $\mathbf{a}, \mathbf{b} \in A(G,X, \mu)$ are free, ergodic, does $\kappa_0^{\mathbf{a}}\preceq \kappa_0^{\mathbf{b}}$ imply $\mathbf{a}\preceq  \mathbf{b}$? 
\end{problem}
We provide a negative answer below. The proof is based on a result about invariant random subgroups of $G= \mathbf{F}_\infty$, the free group on a countably infinite set of generators, which might be of independent interest.

If $I$ is a countable set and $\alpha$ is an action of a countable group $G$ on $I$, we will write $\mathbf{s}_\alpha$ for the corresponding $\textbf{generalized shift action}$ on $2^I$with the usual product measure, given by $(\mathbf{s}_\alpha(g) \cdot f)(i) = f(\alpha(g)^{-1} \cdot i)$. If $I = G/H$, for some $H \leq G$, we will write $\tau_{G/H}$ for the left-translation action of $G$ on $G/H$ and $\mathbf{s}_{G/H}$ instead of $\mathbf{s}_{\tau_{G/H}}$. If $H$ is trivial, we write $\mathbf{s}_G$ instead of $\mathbf{s}_{G/H}$.

We also let $\lambda_\alpha$ be the representation on $\ell^2(I)$ given by $(\lambda_\alpha(g) \cdot f)(i) = f(\alpha(g)^{-1} \cdot i)$. Note that $\lambda_{\tau_{G/H}}$ is the usual $\textbf{quasi-regular representation}$ of $G$ on $\ell^2(G/H)$, which we will denote by $\lambda_{G/H}$.

We call a subgroup $H\leq G$ with $[G:H]=\infty$ $\textbf{action-maximal}$ if $ \mathbf{s}_{G/H}$ is action-maximal and $\textbf{representation-maximal}$ if 
 $\lambda_{G/H}$ is representation-maximal. It was shown in \cite{K1} that there are $H$ which are action-maximal and also $H$ which are representation-maximal, for any non-abelian free group $G$.

An $\textbf{invariant random subgroup (IRS)}$ of $G$ is a probability Borel measure on $\textrm{Sub}(G)$, the compact space of subgroups of $G$, which is invariant under the (continuous) action of $G$ on $\textrm{Sub}(G)$ by conjugation. Denote by $\mathcal{M}_G\subseteq \textrm{Sub}(G)$ the set of all $H\leq G$ that are both action-maximal and representation-maximal. We show the following:


\begin{theorem} \label{thm1} Let $G= \mathbf{F}_\infty$. Then there exists an IRS of $G$ which is supported by $\mathcal{M}_G$.
 \end{theorem}
 
 Using this and the result of Dudko-Grigorchuk \cite[Proposition 8]{DG}, we then prove the following:

\begin{theorem} \label{thm2}

Let $G= \mathbf{F}_\infty$. Then there exists a free, ergodic  $\mathbf{a} \in A(G,X, \mu)$ such that $\mathbf{a}$ is not action-maximal but $ \kappa_0^\mathbf{a}$ is representation-maximal. 
\end{theorem}

Let $\mathbf{a}$ be as in Theorem \ref{thm2}. Since $G= \mathbf{F}_\infty$ does not have property (T), the free, ergodic actions $\mathbf{b} \in A(G,X, \mu)$ are dense in $A(G,X, \mu)$ (see \cite[Theorems 12.2 and 10.8]{K}), so there is a free, ergodic $\mathbf{b} \in A(G,X, \mu)$ such that $\mathbf{b}\npreceq \mathbf{a}$. On the other hand $\kappa_0^{\mathbf{b}}\preceq \kappa_0^{\mathbf{a}}$, thus we have a negative answer to Problem \ref{prob}.

We employ below the following notation:

If $\alpha$ is an action of $G$ on $I$ and $S \subseteq G$, we write $\alpha(S) = \{\alpha(g): g \in S\} \subseteq \mathrm{Sym}(I)$. For $G= \mathbf{F}_\infty$, we let $g_0, g_1, \dots $ be free generators of $G$ and let $G_n = \langle g_0, g_1, \dots , g_n\rangle\leq G$.


If $x$ is a real number, we write $\lfloor x \rfloor$ for the largest integer less than or equal to $x$. If $x, y$ are real numbers and $\epsilon > 0$, we write $x \approx_\epsilon y$ to mean $|x-y| < \epsilon$. Finally, $\mathbb{N} = \{0,1,2, \dots \}$ and $\mathbb{N}^+ = \{1,2,3, \dots \}$

$\textbf{For the rest of the paper}, G=\mathbf{F}_\infty$.

\section{Proof of Theorem \ref{thm1}}

The structure of the proof is as follows. In Subsection \ref{subsec1} we state three lemmas. Temporarily assuming these lemmas, in Subsection \ref{subsec2} we give the main argument establishing Theorem \ref{thm1}. Then in Subsection \ref{subsec3} we prove the lemmas from Subsection \ref{subsec1}.

Recall that for $\mathbf{a} \in A(G,X,\mu)$, we have $\mathbf{a} \preceq \mathbf{b}$ if and only if $\mathbf{a}$ lies in the closure of the isomorphic copies of $\mathbf{b}$. In particular, $\mathbf{b}$ is action-maximal if and only if the isomorphic copies of $\mathbf{b}$ are dense in $A(G,X,\mu)$. We will use these equivalences without comment several times in the sequel.

\subsection{Statements of lemmas} \label{subsec1}

The first lemma provides a general method for constructing invariant random subgroups.

\begin{lemma} \label{lem1} Let $\alpha$ be an action of $G$ on a countably infinite set $I$. Suppose there is an increasing sequence of non-empty finite subsets $(F_n)_{n=0}^\infty$ of $I$ such that $\bigcup_{n=0}^\infty F_n = I$ and $F_n$ is $\alpha(G_n)$-invariant. Let $\theta_n$ be the probability measure on $\mathrm{Sub}(G)$ given by the pushfoward of the uniform measure on $F_n$ under the map $v \mapsto \mathrm{stab}_\alpha(v)$ (where $\mathrm{stab}_\alpha(v)$ is the stabilizer of $v$ in $\alpha$). Let $\theta$ be any weak-star limit point of the $\theta_n$. Then $\theta$ is an invariant random subgroup of $G$. \end{lemma}

In order to state the second lemma, we need the following definition.

\begin{definition} Let $\alpha$ be an action of $G$ on a finite set $V$ and let $n$ be such that all $\alpha(g_k), k>n,$ act trivially. Let $\beta$ be an action of $G$ on a countably infinite set $I$. Let $Q \subseteq I$ be a finite set. We will say that $\alpha$ (relative to $n$) \textbf{appears in} $\beta$ \textbf{within} $Q$ if there is a $\beta(G_n)$-invariant set $W \subseteq Q$ and a bijection $\phi:V \to W$ such that $\phi(\alpha(g) \cdot v) = \beta(g) \cdot \phi(v)$ for all $v \in V$ and $g \in G_n$. We will say that $\alpha$ \textbf{appears in} $\beta$ if it appears within some finite subset of $I$. \end{definition}

Note that if $\alpha$ appears in $\beta$ as above, then $\mathbf{s}_{\alpha \upharpoonright G_n}$ is a factor of $\mathbf{s}_{\beta \upharpoonright G_n}$.

\begin{lemma} \label{lem2} There exists a sequence of finite sets $(V_n)_{n=1}^\infty$, with ${|V_n|}\to \infty$,  and actions $(\alpha_n)_{n=1}^\infty$ of $G$, where $\alpha_n$ acts transitively on $V_n$ so that all $g_k, k>n,$ act trivially in $\alpha_n$, such that if $\beta$ is a transitive action of $G$ on a countably infinite set and $\alpha_n$ (relative to $n$) appears in $\beta$ for each $n$, then $\mathbf{s}_\beta$ is action-maximal and $\lambda_\beta$ is representation-maximal.
 \end{lemma}

Fix a sequence of finite sets $V_n$ and actions $\alpha_n$ of $G$ on $V_n$, $n\geq 1$, as in Lemma \ref{lem2}. Given $f: \mathbb{N} \to \mathbb{N^+}$, $m>0$, write $C_m(f) = \sum_{n=0}^{m-1} (|V_{f(n)}|+1).$ We will need a function $f$ with the following properties. 

\begin{lemma} \label{lem7} There exists a function $f:\mathbb{N} \to \mathbb{N}^+$ such that:

\begin{enumerate}[(i)] \item for every $n\geq 1$ there exists positive integer $K=K_n$ such that for all $j$ there is $l$ with $\left \lfloor \frac{j}{K} \right \rfloor =  \left \lfloor \frac{l}{K} \right \rfloor$ and $f(l) = n$, \label{item1.2}
\item for every $\epsilon > 0$, there exists $t>0$, such that for all $m>0$ we have  \[ \frac{1}{C_m(f)} \sum_{n =1}^t (|V_n|+1) \cdot \bigl \vert \bigl \{j \in \{0,\ldots, m-1 \}: f(j) = n \bigr \} \bigr \vert > 1- \epsilon. \label{eq4.0} \] \label{item1.3}
\end{enumerate}  \end{lemma}

\subsection{Main argument} \label{subsec2}

Let, for $n\geq 1$, $\alpha_n$ and $V_n$ be as in Lemma \ref{lem2} and let $f$ be as in Lemma \ref{lem7}. Choose a pairwise disjoint sequence of finite sets $W_n$, $n\geq 0$, such that $|W_n| = |V_{f(n)}|$. Define an action of $\alpha$ of $G$ on $\bigcup_{n=0}^\infty W_n$ by identifying $W_n$ with $V_{f(n)}$ and letting $G$ act on $W_n$ according to $\alpha_{f(n)}$. Let $\{u_n\}_{n=0}^\infty$ be an enumeration of a countably infinite set disjoint from the $W_n$. We now modify $\alpha$ to obtain a new action $\beta$ of $G$ on $I = \left( \bigcup_{n=0}^\infty W_n \right) \cup \{u_n\}_{n=0}^\infty$. We will have that $\beta(g_k)$ agrees with $\alpha(g_k)$ on $W_n$ when $k \in \{0,\ldots,f(n)\}$. 

For each $n$, choose a point $w_n \in W_n$ and let $\beta(g_{f(n) + 1})$ transpose $w_n$ with $u_n$. Let $(l_n)_{n=0}^\infty$ be a strictly increasing sequence of indices such that $\max(n,f(0),\ldots,f(n+1)) +1 < l_n $. Let $\beta(g_{l_n})$ transpose $w_n$ and $w_{n+1}$.

Fix $n \geq 1$. We now define how $\beta(g_n)$ acts on $\{u_j\}_{j=0}^\infty$. For $k \in \mathbb{N}$, consider the discrete interval \[ D_{n,k} = \{k \cdot n,\ldots,(k+1) \cdot n - 1\}.\] We would like to have $\beta(g_n)$ make a cycle out $\{u_j, j \in D_{n,k}\}$ for each $k$. Unfortunately, we cannot achieve that exactly since there may by some $j \in D_{n,k}$ for which $f(j) +1 = n $, and in this case we will have already used $g_n$ to link $W_j$ with $u_j$. Thus for each $k$, we will let $\beta(g_n)$ make a cycle out of the set \[ \{u_j: j \in D_{n,k} \mbox{ and } f(j) +1 \neq n\}, \] making no modification to the action of $\beta(g_n)$ on those $u_j$ for which $f(j)+1 = n$. We will call these cycles the top cycles of $\beta(g_n)$. We have the following picture of $\beta$, where $n=f(3)+1=6$ and we consider the interval $D_{6,0}$.
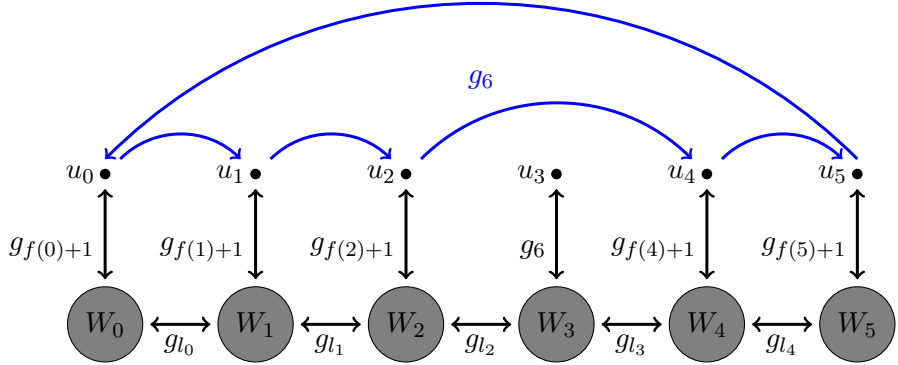
\begin{center}

\vspace{0.25 in}

\begin{tikzpicture}

\filldraw [fill=gray] (0.0,0) circle [radius=0.5 cm];
\node at (0.0,0) {$W_0$};

\draw[line width = 0.4 mm ,<->] (0,0.6) -- (0,1.8);
\node[left] at (0,1) {$g_{f(0)+1}$};

\draw[line width = 0.4 mm, <->] (0.6,0) -- (1.4,0);
\node[below] at (1,0) {$g_{l_0}$};

\fill (0.0,2) circle (2pt);
\node[left] at (0.0,2) {$u_0$};

\draw[blue,line width = 0.4 mm ,->] (0.2,2.2) to [bend left = 45] (1.8,2.2);

\filldraw [fill=gray] (2.0,0) circle [radius=0.5 cm];
\node at (2.0,0) {$W_1$};

\draw[line width = 0.4 mm ,<->] (2,0.6) -- (2,1.8);
\node[left] at (2,1) {$g_{f(1)+1}$};

\draw[line width = 0.4 mm, <->] (2.6,0) -- (3.4,0);
\node[below] at (3,0) {$g_{l_1}$};

\fill (2.0,2) circle (2pt);
\node[left] at (2.0,2) {$u_1$};

\draw[blue,line width = 0.4 mm ,->] (2.2,2.2) to [bend left = 45] (3.8,2.2);

\filldraw [fill=gray] (4.0,0) circle [radius=0.5 cm];
\node at (4.0,0) {$W_2$};

\draw[line width = 0.4 mm ,<->] (4,0.6) -- (4,1.8);
\node[left] at (4,1) {$g_{f(2)+1}$};

\draw[line width = 0.4 mm, <->] (4.6,0) -- (5.4,0);
\node[below] at (5,0) {$g_{l_2}$};

\fill (4.0,2) circle (2pt);
\node[left] at (4.0,2) {$u_2$};

\draw[blue,line width = 0.4 mm ,->] (4.2,2.2) to [bend left = 45] (7.8,2.2);

\node[blue,above] at (5,3) {$g_6$};

\filldraw [fill=gray] (6.0,0) circle [radius=0.5 cm];
\node at (6.0,0) {$W_3$};

\draw[line width = 0.4 mm ,<->] (6,0.6) -- (6,1.8);
\node[left] at (6,1) {$g_6$};

\draw[line width = 0.4 mm, <->] (6.6,0) -- (7.4,0);
\node[below] at (7,0) {$g_{l_3}$};

\fill (6.0,2) circle (2pt);
\node[left] at (6.0,2) {$u_3$};

\filldraw [fill=gray] (8.0,0) circle [radius=0.5 cm];
\node at (8.0,0) {$W_4$};

\draw[line width = 0.4 mm ,<->] (8,0.6) -- (8,1.8);
\node[left] at (8,1) {$g_{f(4)+1}$};

\draw[line width = 0.4 mm, <->] (8.6,0) -- (9.4,0);
\node[below] at (9,0) {$g_{l_4}$};

\fill (8.0,2) circle (2pt);
\node[left] at (8.0,2) {$u_4$};

\draw[blue,line width = 0.4 mm ,->] (8.2,2.2) to [bend left = 45] (9.8,2.2);

\draw[blue,line width = 0.4 mm ,<-] (0.0,2.2) to [bend left = 45] (10,2.2);

\filldraw [fill=gray] (10.0,0) circle [radius=0.5 cm];
\node at (10.0,0) {$W_5$};

\draw[line width = 0.4 mm ,<->] (10,0.6) -- (10,1.8);
\node[left] at (10,1) {$g_{f(5)+1}$};


\fill (10.0,2) circle (2pt);
\node[left] at (10.0,2) {$u_5$};





(13.8,2.2);








\end{tikzpicture}
\end{center}
\vspace{0.5 in}

Finally $\beta$ is defined trivially for all other points. Clearly $\beta$ acts transitively. Write for $m>0$, $\left( \bigcup_{k=0}^{m-1} W_k \right) \cup \{u_0,\ldots,u_{m-1}\} = T_m$ and for $m\geq 0$, $T_{m!} = F_m$. Thus $F_m$ is invariant under $\beta(G_m)$. For each $m$, define a measure $\theta_m$ on $\mathrm{Sub}(G)$ be letting $\theta_m$ be the pushforward of the uniform measure on $F_m$ under the map $v \mapsto \mathrm{stab}_\beta(v)$. Let $\theta$ be a weak-star limit point of $\theta_m$. By Lemma \ref{lem1}, $\theta$ is an invariant random subgroup of $G$. 

We claim that $\theta$ is supported on $\mathcal{M}_G$. Let $(Q_k)_{k=0}^\infty$ be an increasing sequence of finite subsets of $G$ with $\bigcup_{k=0}^\infty Q_k = G$. For $H \leq G$, let $Q_k/H = \{gH: g \in Q_k\}$. Write, for $n\geq 1, k\in \mathbb{N}$, \[ A_{n,k} = \bigl \{H \leq G: \alpha_n \mbox{ appears in }\tau_{G/H} \mbox{ within } Q_k/H \bigr \}. \] By definition, if $H \in \bigcup_{k=0}^\infty A_{n,k}$, then $\alpha_n$ appears in $\tau_{G/H}$. Therefore by Lemma \ref{lem2}, we have \[ \bigcap_{n=1}^\infty \bigcup_{k=0}^\infty A_{n,k} \subseteq \mathcal{M}_G.\] Thus it suffices to show that for each $n\geq 1 $ we have $\sup_{k < \infty} \theta(A_{n,k}) =1$. Fix $n$ and $\epsilon > 0$. Since the set $A_{n,k}$ is clopen for each $k$, it is enough to show the following:

\begin{claim} \label{lem3} There is some $k \in \mathbb{N}$, such that for all $m>0$, we have $\theta_m(A_{n,k}) > 1- \epsilon$.\end{claim}

Let $t$ be large enough that Lemma \ref{lem7}(\ref{eq4.0}) holds for our chosen $\epsilon.$ We now define five finite subsets of $G$. 

\begin{itemize} \item Let $S_1 \subseteq G$ consist of $\{1_G\}$ together with every word in the generators $g_0,\ldots,g_t$ with length at most $\max_{1 \leq j \leq t} |V_j|$. If $f(j) \leq t$, this choice will allow us to pass between points in $W_j$ using an element of $S_1$.

\item Let $S_2 = \{1_G, g_0,\ldots,g_{t+1} \}$. If $f(j) \leq t$, this choice will allow us to pass to $u_j$ from some point in $W_j$ using an element of $S_2$.

\item Let $S_3$ consist of all words in the generators $g_K,g_{2K},g_{3K}$ of length at most $3K$, where $K = K_n$ is the number provided by Lemma \ref{lem7}(\ref{item1.2}) for our fixed $n$. We will explain this choice later.

\item Let $S_4 = \{g_{n+1}\}$. If $f(l) = n$, we will use $g_{n+1}$ to pass from $u_l$ to some point in $W_l$.

\item Let $S_5$ consists of all words in the generators $g_1,\ldots,g_n$ of length at most $|V_n|$. If $f(l) =n$, this choice will allow us to pass between any two points of $W_l$ using an element of $S_5$. \end{itemize}

Let $k$ be large enough that $Q_k$ contains $S_5 \cdot  S_4 \cdot S_3 \cdot S_2 \cdot S_1$. We assert that the following implies Claim \ref{lem3}.

\begin{claim} \label{cla1} If $v \in W_j \cup \{u_j\}$ and $f(j) \leq t$, then $\alpha_n$ appears in $\tau_{G/\mathrm{stab}_\beta(v)}$ within $Q_k/\mathrm{stab}_\beta(v)$.
 \end{claim}

Indeed, suppose Claim \ref{cla1} holds and let $m >0$. Note that $C_{m!}(f)$ defined as in Lemma \ref{lem7} is exactly $|T_{m!}|$. Thus we have \begin{align} \theta_m(A_{n,k}) & = \frac{1}{|T_{m!}|} \cdot \bigl \vert \bigl \{v \in  T_{m!}\colon \mathrm{stab}_\beta(v) \in A_{n,k} \bigr \} \bigr \vert \label{eq9.0} \\ & \geq \frac{1}{|T_{m!}|} \cdot \bigl \vert \bigl \{ v \in T_{m!}: v \in W_j \cup \{u_j\} \mbox{ and }f(j) \leq t \bigr \} \bigr \vert \label{eq9.1} \\ & =  \frac{1}{|T_{m!}|} \sum_{n =1}^t (|V_n|+1) \cdot \bigl \vert \bigl \{j \in \{0,\ldots, m!-1 \}: f(j) = n \bigr \} \bigr \vert \label{eq9.2} \\ & > 1- \epsilon \label{eq9.3}, \end{align}

where \begin{itemize} \item (\ref{eq9.0}) follows from the definition of $\theta_m$, \item (\ref{eq9.1}) follows from (\ref{eq9.0}) by Claim \ref{cla1}, \item (\ref{eq9.2}) follows from (\ref{eq9.1}) since $|W_j| = |V_{f(j)}|$,  \item (\ref{eq9.3}) follows from (\ref{eq9.2}) by Lemma \ref{lem3}(\ref{item1.3}). \end{itemize} Thus it remains to establish Claim \ref{cla1}.

Fix $j$ with $f(j) \leq t$. By our choice of $K$, there is some $l$ such that $\lfloor j/K \rfloor = \lfloor l/K \rfloor$ and $f(l) = n$. Fix $v \in W_j \cup \{u_j\}$. Write $H = \mathrm{stab}_\beta(v)$ and let $P = \{gH: \beta(g) \cdot v \in W_l\}$. Since $\beta(G_n)$ acts on $W_l$ according to $\alpha_n$, it follows that $\alpha_n$ appears in $\tau_{G/H}$ within $P$. Therefore it is enough to show that $P \subseteq Q_k/H$, or equivalently $W_l \subseteq \beta(Q_k) \cdot v$. The idea is that we have chosen $k$ large enough that we can reach any point in $W_l$ from $v$ using the $\beta$ action of a word from $Q_k$.

By our choice of $S_1$, if $v\in W_j$ there is an element $\gamma \in S_1$ such that $\beta(\gamma) \cdot v = w_j$ where $w_j$ is the point in $W_j$ connected to $u_j$. The connection between $w_j$ and $u_j$ is made by $\beta(g_{f(j)+1})$. We have $g_{f(j) +1} \in S_2$ since $f(j) \leq t$.  Thus $u_j = \beta(\gamma)\cdot v$, where $\gamma \in S_2 \cdot S_1$.

Note that our assumption on $l$ guarantees that $l$ lies between the same pair of multiples of $K$ as $j$ does. We would like to say that this allows us to pass from $u_j$ to $u_l$ using $\beta(g_K)^i$ for some $i \in [-K, K]$. However, there is the minor issue of the points $u_d$ which are skipped the top cycles of $\beta(g_K)$. We can easily overcome this obstacle by noting that for any $d$, at most one of $\beta(g_K), \beta(g_{2K})$ and $\beta(g_{3K})$ skips over $u_d$, and therefore there is a word $\gamma'$ in $g_K,g_{2K},g_{3K}$ of length at most $3K$ such that $\beta(\gamma') \cdot u_j = u_l$. We have $\gamma' \in S_3$. 

Since $f(l) = n$, we see that $u_l$ is connected to $W_l$ by $\beta(g_{f(l)+1}) = \beta(g_{n+1})$. Therefore $\beta(g_{n+1} \gamma' \gamma) \cdot v \in W_l$. Since $W_l \subseteq \beta(S_5) \cdot \beta(g_{n+1} \gamma' \gamma)\cdot v$, we have that $W_l \subseteq \beta(Q_k) \cdot v$ and we are done.

\subsection{Proofs of lemmas} \label{subsec3}

\begin{proof}[Proof of Lemma \ref{lem1}] Let $h_1,\ldots,h_l,k_1,\ldots,k_{l'},g \in G$ and let $\epsilon > 0$. Let $m$ be large enough that $h_1,\ldots,h_l,$ $k_1,\ldots,k_{l'},g$ are words in the generators $\{g_0,\ldots,g_m\}$. Write \[ C = \{H \leq G: h_1,\ldots,h_l \in H \mbox{ and } k_1,\ldots,k_{l'} \notin H \}.  \] Note that $C$ is a clopen set and therefore there is some $n \geq m$ such that \begin{equation} \theta(C) \approx_\epsilon \theta_n(C) \mbox{ and } \theta(g C g^{-1}) \approx_\epsilon \theta_n(g C g^{-1}). \label{eq1}  \end{equation}

Noting that $F_n$ is $\alpha(\langle g,h_1,\ldots,h_l,k_1,\ldots,k_{l'}\rangle)$ invariant we have \begin{align*} \theta_n(g C g^{-1}) &= \frac{1}{|F_{n}|} \cdot \bigl \vert \bigl \{v \in F_{n}:  \alpha(g h_j g^{-1}) \cdot v = v \mbox{ for all } j \in \{1,\ldots,l\}  \\ &  \mbox{ and } \alpha(g k_j g^{-1}) \cdot v \neq v \mbox{ for all }j \in \{1,\ldots,l'\} \bigr \} \bigr \vert  \\  &= \frac{1}{|F_{n}|} \cdot \bigl \vert  \bigl \{v \in F_{n}:  \alpha(h_j) \alpha(g^{-1}) \cdot v = \alpha(g^{-1}) \cdot v \mbox{ for all } j \in \{1,\ldots,l\}  \\ & \mbox{ and } \alpha(k_j) \alpha(g^{-1}) \cdot v \neq \alpha(g^{-1}) \cdot v \mbox{ for all }j \in \{1,\ldots,l' \} \bigr \} \bigr \vert \\ & = \frac{1}{|F_{n}|} \cdot \bigl \vert  \bigl \{w \in F_{n}:  \alpha(h_j) \cdot w = w \mbox{ for all } j \in \{1,\ldots,l\}  \\ &\mbox{ and } \alpha(k_j) \cdot w \neq w \mbox{ for all }j \in \{1,\ldots,l' \} \bigr \} \bigr \vert \\ & = \theta_n(C)  \end{align*}

Then from (\ref{eq1}) we have $\theta(C) \approx_{2 \epsilon} \theta(g C g^{-1})$. \end{proof}

\begin{proof}[Proof of Lemma \ref{lem2}]
It is clearly enough to find such $V_n, \alpha_n$ such that for any $\beta$ as in that lemma, $\mathbf{s}_\beta$ is action-maximal and another sequence, also denoted below by  $V_n, \alpha_n$, such that for any $\beta$ as in that lemma, $\lambda_\beta$ is representation-maximal. Then by interlacing these two sequences, we have a sequence that achieves both goals.

{\it Case 1}: We first find the sequence for which the appropriate $\mathbf{s}_\beta$ is action-maximal. By \cite[Theorem 5.1]{K1}, there is a countably infinite set $J$ and a transitive action $\alpha$ of $G$ on $J$ such that $\mathbf{s}_\alpha$ is action-maximal. Identify $(X,\mu)$ with $2^J$ carrying the usual product measure. For a finite set $T \subseteq J$ and $\rho \in 2^T$, write
 \[ 
 N_\rho = \bigl \{x \in 2^J: x(v) = \rho(v) \mbox{ for all } v \in T \bigr \}.
 \]
  For $n \geq 1$, $\epsilon > 0$ and a finite set $T \subseteq J$, let $U_{n,\epsilon,T}$ be the set of all $\mathbf{c} \in A(G,X,\mu)$ such that 
 \[
\mu \bigl(\mathbf{s}_\alpha(g_k) \cdot N_{\rho} \cap N_{\sigma} \bigr) \approx_\epsilon \mu\bigl(\mathbf{c}(g_k) \cdot N_{\rho} \cap N_\sigma \bigr),
  \forall \sigma,\rho \in 2^T, \ k \in \{0,\ldots,n-1\}  .
 \]

Observe that the collection of all $U_{n,\epsilon,T}$ is a neighborhood basis at $\mathbf{s}_\alpha \in A(G,X,\mu)$. Let $(T_n)_{n=1}^\infty$ be an increasing sequence of finite subsets of $J$ with $\bigcup_{n=1}^\infty T_n = J$. Write $U_n = U_{n,2^{-n-|T_n|},T_n}$. Then the sets $U_n$ form a neighborhood basis at $\mathbf{s}_\alpha$. Note that for each $n\geq 1$ and each $k \in \{0,\ldots,n-1\}$, we can extend $\alpha(g_k) \upharpoonright \left ( T_n \cup \bigcup_{j=0}^{n-1} \alpha(g_j) \cdot T_n \right)$ to a permutation of $J$ which is trivial on the complement of a finite set containing $T_n \cup\bigcup_{j=0}^{n-1} \alpha(g_j) \cdot T_n$. Hence for each $n\geq 1$, we can find an action $\widehat{\alpha}_n$ of $G$ on $J$ with the following properties:

\begin{enumerate}[(I)] \item $\widehat{\alpha}_n(g_k) \cdot v = \alpha(g_k) \cdot v$, if $k \in \{0,\ldots,n-1\}$ and $v \in T_n$. \label{item1}
\item $\widehat{\alpha}_n(g_k)$ acts trivially if $k > n$. \label{item2}
\item There is a $\widehat{\alpha}_n$-invariant finite set $V_n \subseteq J$ such that $\widehat{\alpha}_n \upharpoonright (J \setminus V_n)$ is trivial and $\widehat{\alpha}_n \upharpoonright V_n$ is transitive. \label{item3} \end{enumerate}

By (\ref{item1}) we see that $\mathbf{s}_{\widehat{\alpha}_n}(g_k) \cdot N_\rho = \mathbf{s}_{\alpha}(g_k) \cdot N_\rho$ for all $\rho \in 2^{T_n}$ and $k \in \{0,\ldots,n-1\}$. Therefore $\mathbf{s}_{\widehat{\alpha}_n} \in U_n$.  Write $\alpha_n = \widehat{\alpha}_n \upharpoonright V_n$. By (\ref{item2}) all $g_k, k>n$, act trivially in $\alpha_n$. Observe that (\ref{item3}) implies that $\mathbf{s}_{\widehat{\alpha}_n} \cong \mathbf{s}_{\alpha_n} \times \boldsymbol \iota$, where $\boldsymbol \iota$ is the trivial action of $G$ on a nonatomic standard probability space. Thus for each $n\geq 1$ there is an isomorphic copy of $\mathbf{s}_{\alpha_n} \times \boldsymbol \iota$ in $U_n$.

Suppose $\beta$ is a transitive action of $G$ on a countably infinite set such that $\alpha_n$ appears in $\beta$ for each $n\geq 1$. Note that $\boldsymbol{s}_\beta$ is ergodic (see, e.g., \cite[2.1]{KT}). Then $\mathbf{s}_{\alpha_n \upharpoonright G_n }$ is a factor of $\mathbf{s}_{\beta \upharpoonright G_n}$ and hence $\mathbf{s}_{\alpha_n \upharpoonright G_n} \times (\boldsymbol \iota\upharpoonright G_n) $ is a factor of $\mathbf{s}_{\beta \upharpoonright G_n} \times (\boldsymbol \iota\upharpoonright G_n) $. Using the fact that the definition of $U_n$ depends only on $G_n$, this implies that for each $n\geq 1$ there is an isomorphic copy of $\mathbf{s}_\beta \times \boldsymbol \iota $ in $U_n$. Therefore there is a sequence of isomorphic copies of $\mathbf{s}_\beta \times \boldsymbol \iota$ in $A(G,X,\mu)$ which converges to $\mathbf{s}_\alpha$. Since the isomorphic copies of $\mathbf{s}_\alpha$ are dense in $A(G,X,\mu)$, this implies that the isomorphic copies of $\mathbf{s}_\beta \times \boldsymbol \iota$ are dense in $A(G,X,\mu)$. 

By \cite[Theorem 3.11]{T}, we see that any ergodic action $\mathbf{d}$ of $G$ is weakly contained in almost every ergodic component of $\mathbf{s}_\beta \times \boldsymbol \iota$. In particular, any ergodic action $\mathbf{d}$ of $G$ is weakly contained in $\mathbf{s}_\beta$ and therefore the isomorphic copies of $\mathbf{s}_\beta$ are dense in the ergodic actions. Since $G$ does not have Property ($\mathrm{T}$), \cite[Theorem 12.2]{K} implies that the isomorphic copies of $\mathbf{s}_\beta$ are dense in $\mathrm{A}(G,X,\mu)$.

{\it Case 2}: We next find a sequence $V_n, \alpha_n$, for which the appropriate $\lambda_\beta$ is representation-maximal. We start with a transitive action $\alpha$ of $G$ on a countably infinite set $J$ such that $\lambda_\alpha$ is representation-maximal (see \cite[Theorem 5.5]{K1}. Then proceed as in the proof of Case 1 to find $V_n, \alpha_n$ such that for some isomorphic copy $\sigma_n$ of $\lambda_{\alpha_n}\oplus\infty 1_G$, $(\sigma_n)$ converges to $\lambda_\alpha$, where $1_G$ is the trivial one-dimensional representation of $G$ and $\infty 1_G$ the direct sum of countably many copies of $1_G$, i.e., the trivial representation on a separable, infinite-dimensional Hilbert space. Let now $\beta$ be as above. Then the isomorphic copies of $\lambda_\beta\oplus\infty 1_G$ converge to $\lambda_\alpha$. By a result of Hjorth, see \cite[H.7]{K}, the irreducible representations are dense in $\textrm{Rep}(G,H)$. Every irreducible representation $\pi$ is $\preceq\lambda_\alpha$ and thus $\preceq_Z\lambda_\alpha\preceq_Z \lambda_\beta \oplus \infty 1_G$, where $\preceq_Z$ is weak containment in the sense of Zimmer. Recall that $\sigma\preceq_Z\rho$ iff $\sigma$ is in the closure of the isomorphic copies of $\rho$. Also $\sigma\preceq_Z \rho \implies \sigma\preceq \rho$ and for $\sigma$ irreducible, $\sigma\preceq_Z \rho \iff \sigma\preceq \rho$ (see \cite[page 397]{BHV} and \cite[page 209]{K}). Then by \cite[Proposition 3.5]{AE} $\pi$ is a subrepresentation of an ultrapower of $\lambda_\beta\oplus \infty 1_G$, which is of course of the form $\lambda^*_\beta\oplus \eta^*$, where $\lambda^*$ is an ultrapower of $\lambda_\beta$ and $\eta^*$ a trivial representation of $G$ on a Hilbert space $H^*$. Let $H_1$ be the space on which this subrepresentation acts, which is a $G$-invariant subspace of the direct sum of the space of $\lambda_\beta^*$ and $H^*$. Then if $v\in H^*$ and $v_1$ is its projection on $H_1$, $v_1$ is $G$-invariant, so as $\pi$ is irreducible, $v_1=0$, i.e., $H^*\perp H_1$. Thus $H_1$ is contained in the space of $\lambda_\beta^*$, i.e., $\pi$ is a subrepresentation of $\lambda_\beta^*$, so $\pi\preceq_Z\lambda_\beta$. Thus the isomorphic copies of $\lambda_\beta$ are dense in $\textrm{Rep}(G,H)$, i.e., $\lambda_\beta$ is representation-maximal.
 \end{proof}

\begin{proof}[Proof of Lemma \ref{lem7}]
Note that letting for $n\geq 1$, $A_n = f^{-1} (\{n\})$ the statement of the lemma is equivalent to the existence of a partition $\mathbb{N} = \bigsqcup_{n\geq 1}A_n$ with the following properties:

(i) For each $n\geq 1$ there is positive integer $K_n$ such that $A_n$ intersects each interval $I^n_i = [iK_n, (i+1)K_n), i = 0,1,2, \dots$.

(ii) Let $g:\mathbb{N}^+ \to \mathbb{N}^+ $ be defined by $g(n) = |V_n| +1$, where $V_n$ is as in Lemma \ref{lem2}. Then we have that for each $\epsilon >0$, there is $t>0$, such that for all $m>0$:

\[
\frac{\sum_{n>t}(|A_n\cap m|\cdot g(n))}{\sum_n(|A_n\cap m|\cdot g(n))} < \epsilon,
\]
where we identify here $m$ with $\{0,1, \dots , m-1\}$.

To construct $A_n, K_n$, first chose $a_2 < a_3 < \dots$ to be large enough so that $a_n$ is divisible by 3 and 
\[
\sum^{\infty}_{n=2}\frac{1}{a_2 \cdots a_n}<\frac{1}{3} {\textrm{ and }}  \frac{a_n}{3}> \frac{g(n)2^n}{g(n-1)}.
\]
We let $A'_1 =\{2i\colon i\in \mathbb{N}\}$ and also put $K_1= 2, K_n = 2 a_2 \cdots a_n$ for $n\geq 2$. We will then inductively define pairwise disjoint $A_2, A_3, \dots$,  which are also disjoint from $A'_1$, to satisfy (ii) above and so that for $n\geq 2$, $A_n$ has exactly one member in each interval $I_i^n$ as above, and finally we put $A_1 = \mathbb{N}\setminus \bigcup_{n=2}^\infty  A_n$.

So assume that $A'_1, A_2, \dots , A_{n-1}$ have been constructed (this is just $A'_1$, if $n= 2$). To find $A_n$, so that (i) above is satisfied, it is enough to have for each $i$, 
\[
K_n > \frac{3}{2} \left|(A'_1 \cup A_2\cup \cdots A_{n-1}) \cap I^n_i\right|.
\]
But
\[
 \left|(A'_1 \cup A_2\cup \cdots A_{n-1}) \cap I^n_i\right| = a_2 \cdots a_n + a_3\cdots a_n+ \cdots + a_{n-1} a_n + a_n,
 \]
 so this follows from $\sum^{\infty}_{n=2}\frac{1}{a_2 \cdots a_n}<\frac{1}{3}$. Also for $i=0$, we can choose the element of $A_n$ in $[0,K_n)$ to be $\geq \frac{K_n}{3}$.

We finally check that (ii) is satisfied. Fix $\epsilon > 0$ and choose $t>1$ so that $\sum_{n>t}^\infty  2 ^{-n}< \epsilon$. Consider now any $m>0$ and $n>t$. 

{\it Case 1}. $m\geq K_n$. Then for some $s>1$, we have that $m\in I^n_{s-1}$ and $|A_n\cap m |\leq s$, while 
\[
\sum_n |A_n\cap m|\cdot g(n)\geq |A_{n-1}\cap m|\cdot g(n-1) \geq (s-1) a_n\cdot g(n-1)
\]
so
\[
\frac{|A_n\cap m|\cdot g(n)}{\sum_n|A_n\cap m|\cdot g(n)}\leq \frac{s\cdot g(n)}{(s-1)\cdot g(n-1)}\cdot\frac{1}{a_n} < 2^{-n}.
\]

{\it Case 2}. $m < K_n$. Then either $m\leq\frac{K_n}{3}$ and $|A_n\cap m|=0$ or $m>\frac{K_n}{3}$ and $|A_n\cap m|\leq1$, in which case also 
\[
|A_{n-1}\cap m| \geq \frac{a_n}{3}.
\]
So for any $m<K_n$, 
\[
\frac{|A_n\cap m|\cdot g(n)}{\sum_n |A_n\cap m|\cdot g(n)} \leq \frac{g(n)}{(\frac{a_n}{3}) g(n-1)}< 2^{-n}.
\]
Thus for any $n>t$, we have 
\[
\frac{|A_n\cap m|\cdot g(n)}{\sum_n |A_n\cap m|\cdot g(n)} < 2^{-n}
\]
and so
\[
\frac{\sum_{n>t}(|A_n\cap m|\cdot g(n))}{\sum_n(|A_n\cap m|\cdot g(n))} < \epsilon
\] \end{proof}

\section{Proof of Theorem \ref{thm2}}






We note that if $\lambda_{G/H}$ is representation-maximal, then $H$ is not amenable. This is because $1_G\preceq \lambda_{G/H}$ implies $\tau_{G/H}$ is amenable (see \cite[Theorem 1.1]{KT}).

We will use the notion of a random Bernoulli shift over an invariant random subgroup; we refer the reader to \cite[Section 5.3]{T} and \cite[Proposition 45]{AGV} for details. Let $\theta$ be the invariant random subgroup constructed in Theorem \ref{thm1} and let $\mathbf{s}_\theta$ be the $\theta$-random Bernoulli shift. Note that for almost every ergodic component $\mathbf{b}$ of $\mathbf{s}_\theta$, almost all stabilizers of $\mathbf{b}$ lie in $\mathcal{M}_G$ and hence the type of $\mathbf{b}$ is supported on $\mathcal{M}_G$. Fix such an action $\mathbf{b}$. Let $(Y,\nu)$ be the underlying space of $\mathbf{b}$.

For $y \in Y$ write $H_y = \mathrm{stab}_\mathbf{b}(y)$. By \cite[Proposition 8]{DG} we have $\lambda_{G/H_y} \preceq  \kappa_0^\mathbf{b}$ for $\nu$-almost every $y \in Y$. Since the type of $\mathbf{b}$ is supported on $\mathcal{M}_G$, for $\nu$-almost every $y$ we have that $\lambda_{G/H_y}$ is representation-maximal and so $ \kappa_0^\mathbf{b}$ is representation-maximal. Let $\mathbf{a} = \mathbf{b} \times \mathbf{s}_G$. Then $\mathbf{a}$ is free and ergodic and $ \kappa_0^\mathbf{a}$ is representation-maximal. Suppose, toward a contradiction, that $\mathbf{a}$ were action-maximal. 

Let $S \subseteq G^2$ be the collection of all pairs $(g,h)$ such that $\langle g,h \rangle$ is nonamenable. Since $\lambda_{G/H_y}$ is representation-maximal for $\nu$-almost every $y \in Y$, and so $H_y$ is not amenable, we see that $S \cap H_y^2$ is nonempty for $\nu$-almost every $y$. Let $\phi:\mathbb{N} \to S$ be an enumeration of $S$. For $y \in Y$ let $\phi_y = \min \{n: \phi(n) \in H^2_y \}$. Then there is some $k \in \mathbb{N}$ such that $\nu(\{y:\phi_y = k\}) > 0$. Write $A = \{y:\phi_y = k\}$ and let $N$ be the subgroup of $G$ generated by the coordinates of $\phi(k)$. Note that for $y \in A$, we have $N \subseteq H_y$, and so $\mathbf{b} \upharpoonright N$ is trivial on $A$. By \cite[Page 74]{K}, since $\mathbf{a}$ is action-maximal for $G$, we have that $\mathbf{a} \upharpoonright N$ is action-maximal for $N$. Observe that \[ \mathbf{a} \upharpoonright N = (\mathbf{b} \upharpoonright N) \times (\mathbf{s}_G \upharpoonright N) \cong (\mathbf{b} \upharpoonright N) \times (\mathbf{s}_N)^\mathbb{N} \cong (\mathbf{b} \upharpoonright N) \times \mathbf{s}_N.\] So writing $\mathbf{c} = (\mathbf{b} \upharpoonright N) \times \mathbf{s}_N$, we have that $\mathbf{c}$ is action-maximal for $N$. 

By  \cite[Theorem 3.11]{T}, this implies that any ergodic action $\mathbf{d}$ of $N$ is weakly contained in almost every ergodic component of $\mathbf{c}$. Note that if $y \in A$, then $\boldsymbol \iota_{\{y\}} \times \mathbf{s}_N \cong \mathbf{s}_N$ is an ergodic component of $\mathbf{c}$, where by $\boldsymbol \iota_{\{y\}}$ we mean the trivial action of $N$ on the one-point space $\{y\}$. Therefore $\mathbf{d} \preceq \mathbf{s}_N$. Since $N$ does not have property (T), the ergodic actions of $N$ are dense in $A(N,X, \mu)$ (see \cite[12.2]{K}), so the isomorphic copies of $\mathbf{s}_N$ are dense in $A(N,X, \mu)$. But by \cite[Proposition 13.2]{K} this contradicts the fact that $N$ is nonamenable. 

\begin{remark}
For $G=\mathbf{F}_\infty$, let $\mathbf{a}$ be as in Theorem \ref{thm2}. Then for any irreducible $\pi$ we have $\pi\preceq \kappa_0^{\mathbf{a}}$, so 
$\pi\preceq_Z \kappa_0^{\mathbf{a}}$. Thus, as the irreducible representations are dense, $\pi\preceq _Z\kappa_0^{\mathbf{a}}$, for all $\pi$. Thus there is a free ergodic action $\mathbf{b}$ such that $\kappa_0^{\mathbf{b}}\preceq_Z \kappa_0^{\mathbf{a}}$ but $\mathbf{b}\npreceq \mathbf{a}$, which is a somewhat stronger negative answer to Problem \ref{prob}.
\end{remark}


\bigskip
\noindent Department of Mathematics

\noindent California Institute of Technology

\noindent Pasadena, CA 91125

\medskip
\noindent\textsf{pjburton@caltech.edu, kechris@caltech.edu}
\end{document}